\documentclass[11pt]{article}
\usepackage{amsmath,amssymb,amsthm}
\usepackage{times}
\usepackage[utf8]{inputenc}
\usepackage{indentfirst}
\usepackage{graphicx,tikz}
\usepackage[T1]{fontenc}
\usepackage{comment}

\newtheorem{problem}{Problem}

\newtheorem{definition}[problem]{Definition}
\newtheorem{theorem}[problem]{Theorem}
\newtheorem{proposition}[problem]{Proposition}

\title{The Width of Hamming Balls}
\author{Kada K\'{a}lm\'{a}n Williams}
\date{October 30, 2024}

\begin{document}

\maketitle

\begin{abstract}
    The width of a poset is the size of its largest antichain. Sperner's theorem states that $(2^{[n]},\subset)$ is a poset whose width equals the size of its largest layer. We show that Hamming ball posets also have this property. This extends earlier work that proves this in the case of small radii. Our proof is inspired by (and corrects) a result of Harper.
\end{abstract}

\section{Introduction}

Let $P$ be an arbitrary set and $<$ be an irreflexive, transitive relation on pairs of its elements. Then we say that $(P,<)$ is a partially ordered set, or to contract, a {\it poset}.

In a poset $P$, with relation $<$, an element $x$ is {\it covered} by $y$ if $x<y$ without any intermediate $z\in P$, such that $x<z$ and $z<y$. For example, in the poset of rational numbers, $0$ is not covered by any element, by the Archimedean axiom. However, in a finite poset, an intermediate element can only be found finitely many times, and so every relation $x<y$ can be deduced from a {\it chain of covers} $x<z_1$, $z_1<z_2$, $\dots$, $z_{k-1}<z_k$, $z_k<y$.

Let $(P,<)$ be a finite poset. Consider the function $r:P\to \{0,1,\dots\}$ given by
    $$r(x')=\begin{cases} 0\quad \text{if } \{x\in P:x<x'\}=\emptyset, \\
    \max\{r(x):x<x'\}+1 \quad \text{otherwise}.\end{cases}$$
The value $r(x)=k$ is the {\it rank} of $x$. The set of $x\in P$ whose rank is $k$ is {\it layer} $P_k$. Clearly, if $r(y)=k$ and $x<y$, then $r(x)<r(y)$.

Let $(P,<)$ be a poset with $C\subseteq P$ and $A\subseteq P$. If any two elements $x\neq y$ in $C$ are comparable, we say that $C$ is a chain. If any two elements $x\neq y$ in $A$ are incomparable, we say that $A$ is an antichain. Observe that in this case, $|A\cap C|\le 1$.

\begin{definition}
    Let $(P,<)$ be a finite poset. The size of the largest antichain of $P$ is the width of $P$, denoted $w(P)$. The size of the largest layer of $P$ is denoted $\ell(P)$.
\end{definition}

Since a layer of a finite poset $P$ is an antichain, $w(P)\ge \ell(P)$.

\begin{theorem}[Sperner, 1928]
    In the poset $P=(2^{[n]},\subset)$, $w(P)=\ell(P)$.
\end{theorem}

Sperner's proof \cite{Spe} observes that if $P$ is the union of $c$ chains, then $w(P)\le c$ -- a simpler proof of this is given by Greene and Kleitman \cite{GrK}. However, a more general approach to poset width calculation considers a collection of chains that covers $x\in P$ with a certain multiplicity. \cite{Lub}

\begin{definition}
    Let $\mathcal{C}$ be a collection of chains in a finite poset $P$. Suppose that if we sample $C\in \mathcal{C}$ uniformly at random, then for all $x\in P$, the chance that $x\in C$ is $\frac1{|P_{r(x)}|}$. Then we say that $\mathcal{C}$ is a unit flow.
\end{definition}

Applying our observation that an antichain $A$ intersects a chain $C$ in at most a singleton to a unit flow $C$, by linearity of expectation, we have the following property.

\begin{definition}
    Let $P$ be a finite poset. If every antichain $A\subset P$ satisfies
    $$\sum_{x\in A} \frac1{|P_{r(x)}|}\le 1,$$
    then we say that $P$ is a LYM poset \cite{DFr} or that $P$ is a strong Sperner poset \cite{HKL}.
\end{definition}

If $P_k$ is the largest layer of $P$, then for all $x\in P$, $|P_{r(x)}|\le |P_k|$. We deduce that if $P$ is a LYM poset, then $P$ is a Sperner poset.

\begin{theorem}[Yamamoto-Meshalkin-Bollob\'{a}s-Lubell] \label{strspe}
    $(2^{[n]},\subset)$ is a LYM poset.
\end{theorem}

\begin{proof}
    It suffices to find a unit flow, letting the symmetries of $[n]$ act on the chain $\emptyset\subset [1]\subset [2]\subset \dots\subset [n]$.
\end{proof}

Finding a unit flow, in general, is difficult. Hence, it is of benefit for us to know a sufficient condition -- a consequence of the Max Flow Min Cut theorem. \cite{DHW}

\begin{proposition}[Kleitman, 1974] \label{flow}
    Let $(P,<)$ be a finite poset, where the LYM inequality $\sum_{x\in A} \frac1{|P_{r(x)}|}\le 1$ holds for all antichains $A$ contained within two consecutive layers. For such a poset $P$, there is a unit flow. 
\end{proposition}

We shall use this condition to prove that product posets and Hamming balls have the LYM property.

\newpage

\section{Harper's theorem on product posets}

\begin{definition}
    Let $(P,<)$ be a finite poset, and let $P_k$ have size $p_k$. If the fractions $\frac{p_{k+1}}{p_k}$, when defined, are decreasing in $k$, we say that $P$ is log-concave.
\end{definition}

For instance, if $P$ is log-concave, then the subset of $P$ given by layers indexed by an arithmetic progression is also log-concave. \cite{WaY}

\begin{definition}
    Let $P$ and $Q$ be posets. Then the set $P\times Q$ with partial order $(x_P,x_Q)\le (y_P,y_Q)$ if $x_P\le y_P$ and $x_Q\le y_Q$ is called the product poset $P\times Q$.
\end{definition}

For example, $(2^{[n]},\subseteq)$ can be recognized as the $n$-fold product of $(\{\emptyset,\{\emptyset\}\},\subseteq)$.

\begin{theorem}[Harper, 1974] \label{harper}
    Let $P$ and $Q$ be finite posets. If both $P$ and $Q$ are log-concave LYM posets, then their product poset $P\times Q$ is LYM and log-concave.
\end{theorem}

\begin{proof}
    For completeness, we give a proof, while correcting an error made in \cite{Har}.

    Let $P$ and $Q$ consist of layers $P_0,\dots,P_h$ and $Q_0,\dots,Q_\mu$, respectively. By Proposition \ref{flow}, there is a unit flow on $P$ and on $Q$, and particularly from $P_k$ to $P_{k+1}$ and from $Q_l$ to $Q_{l+1}$ ($0\le k<h$, $0\le l<\mu$).

    Every element of $P_k\times Q_l\subseteq P\times Q$ has rank $k+l$, and so 
    $$(P\times Q)_m=\bigcup_{k+l=m} P_k\times Q_l.$$
    Consider a "block poset", whose elements are the "blocks" $P_k\times Q_l$ ($0\le k\le h$, $0\le l\le \mu$) and relations are $P_k\times Q_l\le  P_{k'}\times Q_{l'}$ for $k\le k'$ and $l\le l'$. Each block $P_k\times Q_l$ occurs with multiplicity given by its numbers of elements. In this block poset, by Proposition \ref{flow}, a unit flow can be found if the LYM inequality holds within consecutive layers. This would follow from the inequality
    \footnote{It is this inequality, for a unit flow in the block poset, which \cite{Har} proves incorrectly. Estimating
    $$p_{k-i}p_{k+1-j}q_{l+i}q_{l+j}\le\begin{cases} p_{k+1-i}p_{k-j}q_{l+i}q_{l+j} \text{ if }j\le i, \\ p_{k+1-(i+1)}p_{k-(j-1)}q_{l+(i+1)}q_{l+(j-1)}\text{ if }j>i. \end{cases}$$
    and summing, \cite{Har} notes that "every term of the previous sum gives a unique term of this one", falsely, because the contributions of upper bounds for $j=i-1$ and for $i=j-1$ involve $t-1$ duplicates.}
    $$\frac{\sum_{i=0}^{t-1} |P_{k-i}\times Q_{l+i}|}{|(P\times Q)_{k+l}|}\le \frac{\sum_{i=0}^t |P_{k+1-i}\times Q_{l+i}|}{|(P\times Q)_{k+l+1}|},$$
    where $t$ is a positive integer, $0\le k+l<h+\mu$, and undefined $P_k$ or $Q_l$ are empty.
     
    Given a unit flow on the block poset, a unit flow on $P\times Q$ emerges by replacing $P_k\times Q_l<P_{k+1}\times Q_l$ in the random chain with elements of the respective blocks in proportion to the unit flow from $P_k$ to $P_{k+1}$, and similarly for $P_k\times Q_l<P_k\times Q_{l+1}$.
    
    To establish the requisite bounds, denote $|P_k|=p_k$ and $|Q_l|=q_l$. We require
    $$\left(\sum_{i=0}^{t-1} p_{k-i}q_{l+i}\right)\left(\sum_j p_{k+1-j}q_{l+j}\right)\le \left(\sum_{i=0}^t p_{k+1-i}q_{l+i}\right)\left(\sum_j p_{k-j}q_{l+j}\right),$$
    $$\sum_{i=0}^{t-1}\sum_j p_{k-i}p_{k+1-j}q_{l+i}q_{l+j}\le \sum_{i=0}^t\sum_j p_{k+1-i}p_{k-j}q_{l+i}q_{l+j}.$$
    Observe the dependence on $q_{l+j}$ ($j<0$). Its coefficients on either side are 
    $$\sum_{i=0}^{t-1}p_{k-i}p_{k+1-j}q_{l+i}\quad \text{and}\quad \sum_{i=0}^t p_{k+1-i}p_{k-j}q_{l+i}.$$
    By log-concavity of $P$, $p_{k-i}p_{k+1-j}\le p_{k+1-i}p_{k-j}$ holds if $i\ge 0\ge j$. Hence, if only for the terms involving $q_{l+j}$ ($j<0$), our inequality holds. The same is true for the terms involving $p_{k-j}$ ($j\ge t$). Its coefficients, via summand $(j+1)$, are 
    $$\sum_{i=0}^{t-1}p_{k-i}q_{l+i}q_{l+j+1}=\sum_{i=1}^{t}p_{k+1-i}q_{l+i-1}q_{l+j+1}\quad \text{and}\quad \sum_{i=0}^t p_{k+1-i}q_{l+i}q_{l+j}.$$
    By log-concavity of $Q$, $q_{l+i-1}q_{l+j+1}\le q_{l+i}q_{l+j}$ holds if $i\le t\le j$. Furthermore, there is no term that involves both $q_{l+j}$ ($j<0$) and $p_{k-j}$ ($j\ge t$). Finally, for terms that involve neither, one observes an identity
    $$\left(\sum_{i=0}^{t-1} p_{k-i}q_{l+i}\right)\left(\sum_{j=0}^{t} p_{k+1-j}q_{l+j}\right)= \left(\sum_{i=0}^t p_{k+1-i}q_{l+i}\right)\left(\sum_{j=0}^{t-1} p_{k-j}q_{l+j}\right).$$
    Thus, the requisite bounds are valid. Hence, $P\times Q$ is LYM, having a unit flow.

    In \cite{Har}, it was asserted that $P\times Q$ is log-concave, without proof. Following \cite{Can}, the log-concave condition can be rewritten as
    $$|(P\times Q)_k|^2\ge |(P\times Q)_{k+1}|\cdot |(P\times Q)_{k-1}|,$$
    $$\left(\sum_i p_{k-i}q_i\right)\cdot \left(\sum_j p_{k-j}q_j\right)\ge \left(\sum_i p_{k+1-i}q_i\right)\cdot \left(\sum_j p_{k-1-j}q_j\right),$$
    $$\sum_{i,j} (p_{k-i}p_{k-j}-p_{k+1-i}p_{k-1-j})(q_iq_j-q_{i-1}q_{j+1})\ge 0.$$
    By log-convexity of $P$ and $Q$, these two brackets are both $\ge 0$ if $i\le j$ and both $\le 0$ if $i\ge j$. The log-concavity criterion follows.
\end{proof}

\newpage

\section{Hamming balls are LYM posets}

\begin{definition}
The hypercube graph is given by vertex set $2^{[n]}$, with edges between sets that only differ by a singleton. The family of subsets of graph distance at most $\rho$ from a given set of $p$ elements is called a Hamming ball, denoted $B_\rho[p,n-p]$.
\end{definition}

The following result confirms a conjecture of the author \cite{KKW}, who showed that $B_\rho[p,q]$ is a Sperner poset, provided that $p\ge \frac{1}{27}\rho^3+O(\rho^2)$.

\begin{theorem} \label{ham}
    The Hamming ball $B_\rho[p,q]$ is a LYM poset.
\end{theorem}

\begin{proof}
    Given $p,q,r$, let $B_{i,j}\subset B=B_\rho[p,q]$ contain the subsets of $[p+q]$ arising by removing $i$ elements from $\{1,\dots,p\}$ and then including $j$ elements from $\{p+1,\dots,p+q\}$. The covers of subsets in $B_{i,j}$ are then contained by $B_{i-1,j}$ or $B_{i,j+1}$.

    By Proposition \ref{flow}, to show this to be a LYM poset, it suffices to check the relevant inequality for consecutive layers. Moreover, as in our proof of Theorem \ref{harper}, this reduces to verifying a condition for the block poset given by the blocks $B_{i,j}$ and covers $B_{i,j}<B_{i-1,j}$, $B_{i,j}<B_{i,j+1}$. Indeed, a flow from $B_{i,j}$ to $B_{i-1,j}$ or $B_{i,j+1}$ arises by letting the symmetries of $\{1,\dots,p\}\sqcup \{p+1,\dots,p+q\}$ act on a two-element chain.

    Therefore, like Theorem \ref{harper} and Proposition \ref{flow}, we only have to establish that
    $$\text{capacity}(B_{i,j}\cup B_{i-1,j-1}\cup \dots\cup B_{i-(t-1),j-(t-1)})=\frac{\sum_{l=0}^{t-1} |B_{i-l,j-l}|}{\sum_l |B_{i-l,j-l}|}\le $$
    $$\le \text{capacity}(B_{i,j+1}\cup B_{i-1,j}\cup \dots\cup B_{i-t,j-(t-1)})=\frac{\sum_{l=0}^{t} |B_{i-l,j+1-l}|}{\sum_l |B_{i-l,j+1-l}|}.$$
    Here, we allow $B_{i,j}=\emptyset$ whenever $i+j>r$ or $i<0$ or $j<0$, and we presume these layers and $B_{i,j}$ to be nonempty. We can rearrange to
    $$\frac{\sum_l |B_{i-l,j+1-l}|}{\sum_l |B_{i-l,j-l}|}\ge \frac{\sum_{l=0}^{t} |B_{i-l,j+1-l}|}{\sum_{l=0}^{t-1} |B_{i-l,j-l}|}.$$    
    Call the first fraction $f$ and the second one $\phi$. Notice that $|B_{i,j}|=\binom{p}{i}\binom{q}{j}$, whence $$\frac{|B_{i,j+1}|}{|B_{i,j}|}=\frac{q-j}{j+1},\qquad \frac{|B_{i-1,j}|}{|B_{i,j}|}=\frac{i}{p-i+1}.$$
    These are decreasing in $j$ and increasing in $i$, respectively. It follows that if $B_{i+1,j+2}$ is nonempty, then $\frac{|B_{i-l,j+1-l}|}{|B_{i-l,j-l}|}\ge \frac{q-j}{j+1}$ whenever $l\ge 0$. This implies that $\phi\ge \frac{q-j}{j+1}$. However, $\frac{|B_{i+1,j+2}|}{|B_{i+1,j+1}|}<\frac{q-j}{j+1}$. Hence, if we include $|B_{i+1,j+2}|$ in the numerator and $|B_{i+1,j+1}|$ in the denominator, then the fraction $\phi$ decreases. Just so, if $B_{i-t-1,j-t}$ is nonempty, then $\phi\ge \frac{i-t}{p-(i-t)+1}\ge \frac{|B_{i-t-1,j-t}|}{|B_{i-t,j-t}|}$, and so incrementing $t$ decreases $\phi$. Therefore, for given $f$, the fraction $\phi$ is minimal when its numerator includes the complete range of nonzero subset sizes from the corresponding layer.

    But when the proportion of the covers equals $1$, the inequality is obvious.
\end{proof}

\begin{definition}
    Let $(P,<)$ be a poset and $Q\subseteq P$. Suppose that whenever $x<z<y$ and $x,y\in Q$, also $z\in Q$. Then we say that $Q$, with $<$ on its elements, is a convex subposet of $P$. 
\end{definition}

The smallest convex set containing $B_\rho[p,q]=\bigcup_{i+j\le \rho} B_{i,j}$ is the set $\bigcup_{i,j\le \rho} B_{i,j}$. This poset is also LYM, by the same proof as of Theorem \ref{ham}.

\newpage

\textbf{Acknowledgements.} The author would like to thank Imre Leader for helpful comments and suggestions. The author is financially supported by the Internal Graduate Studentship of Trinity College, Cambridge.

\textsc{Department of Pure Mathematics and Mathematical Statistics, University \
of Cambridge, Wilberforce Road, Cambridge CB3 0WB.} \\ \\
\textit{E-mail address:} \texttt{kkw25@cam.ac.uk}

\end{document}